\documentclass[11pt]{article}

\usepackage{amsmath}
\usepackage{amsfonts}
\usepackage{amssymb}
\usepackage{amsthm}
\usepackage{enumitem}
\usepackage{setspace}
\usepackage{etoolbox}

\newtheoremstyle{mystyle}
{11pt}                          
{11pt}                          
{}                                      
{}                                      
{\bfseries}                     
{.}                                     
{5.5pt}                         
{}                                      

\newtheoremstyle{mystyle2}
{11pt}                          
{11pt}                          
{}                                      
{}                                      
{\bfseries}                     
{}                                     
{5.5pt}                         
{}                                      

\theoremstyle{mystyle}
\newtheorem{theorem}{Theorem}
\newtheorem{definition}{Definition}
\newtheorem{lemma}{Lemma}
\newtheorem{proposition}{Proposition}
\newtheorem{corollary}{Corollary}

\theoremstyle{mystyle2}
\newtheorem*{remark}{Remark}
\newtheorem*{remarks}{Remarks}
\renewenvironment{proof}[1][Proof.]{\vspace{-16.5pt} \begin{trivlist}
        \item[\hskip \labelsep {\bfseries #1}]}{\qed \end{trivlist}}

\setitemize{noitemsep, topsep=5.5pt, parsep=5.5pt, partopsep=0pt}

\allowdisplaybreaks
\appto\normalsize{
        \abovedisplayskip=5.5pt plus 2pt minus 2pt
        \belowdisplayskip=5.5pt plus 2pt minus 2pt
        \abovedisplayshortskip=5.5pt plus 2pt minus 2pt
        \belowdisplayshortskip=5.5pt plus 2pt minus 2pt}
\appto\small{
        \abovedisplayskip=5.5pt plus 2pt minus 2pt
        \belowdisplayskip=5.5pt plus 2pt minus 2pt
        \abovedisplayshortskip=5.5pt plus 2pt minus 2pt
        \belowdisplayshortskip=5.5pt plus 2pt minus 2pt}

\newcommand{\gap}{\vspace{11pt}}
\newcommand{\T}{\mathsf{T}}

\newcommand{\tr}{\operatorname{tr}}
\newcommand{\Aut}{\operatorname{Aut}}

\newcommand{\R}{{\cal R}}
\newcommand{\Rn}{{\cal R}^n}

\newcommand{\Sn}{{\cal S}^n}

\newcommand{\Hn}{{\cal H}^n}

\newcommand{\G}{{\cal G}}
\newcommand{\V}{{\cal V}}
\newcommand{\W}{{\cal W}}
\newcommand{\Z}{{\cal Z}}

\newcommand{\ip}[2]{\left< #1,\, #2 \right>}

\newcommand{\mybar}[1]{\mkern 1.5mu \overline{\mkern -1.5mu #1 \mkern -1.5mu} \mkern 1.5mu}

\title{\bf  
Commutation principles in  Euclidean Jordan algebras and normal decomposition systems}
\author{
        M. Seetharama Gowda\\
        Department of Mathematics and Statistics\\
        University of Maryland, Baltimore County\\
        Baltimore, Maryland 21250, USA\\
        gowda@umbc.edu\\[11pt]
        and\\[11pt]
        Juyoung Jeong\\
        Department of Mathematics and Statistics\\
        University of Maryland, Baltimore County\\
        Baltimore, Maryland 21250, USA\\
        juyoung1@umbc.edu
}

\date{\today}

\begin{document}

\maketitle

\begin{abstract}
The commutation principle of Ramirez, Seeger, and Sossa \cite{ramirez-seeger-sossa} proved in the setting of 
Euclidean Jordan algebras says that when the sum of a Fr\'{e}chet differentiable function $\Theta(x)$ and a 
spectral function $F(x)$ is minimized over a spectral set $\Omega$,  any local minimizer $a$ operator commutes with the 
Fr\'{e}chet derivative $\Theta^{\prime}(a)$. In this  paper, we 
 extend this result to  sets and functions which are (just) 
 invariant under algebra automorphisms. 
We also consider a 
 similar principle in the setting of normal decomposition systems.
\end{abstract}

\vspace{1cm}
\noindent{\bf Key Words:} Euclidean Jordan algebra,  (weakly) spectral sets/functions,  automorphisms,  
 commutation principle, normal decomposition system, variational inequality problem, cone complementarity problem.
\\

\noindent{\bf AMS Subject Classification:}  17C20, 17C30, 52A41, 90C25.
\newpage

\section{Introduction}
      Let $\V$ be a Euclidean Jordan algebra of rank $n$ \cite{faraut-koranyi} and $\lambda:\V\rightarrow \Rn$ denote the eigenvalue map (which takes $x$ to $\lambda(x)$, the vector of eigenvalues of $x$ with entries written in the decreasing order). A  set $\Omega$ in $\V$ is said to be a {\it spectral set} \cite{baes} if it is of the form
$ \Omega = \lambda^{-1}(Q), $
where $Q$ is a permutation invariant set in $\Rn$.
A function $F:\V\rightarrow \R$ is said to be a {\it spectral function} \cite{baes} if it is of the form
$F = f \circ \lambda, $
where $f : \Rn \rightarrow \R$ is a permutation invariant function. 

\gap

Extending an earlier result of Iusem and Seeger \cite{iusem-seeger} for 
real symmetric matrices, Ramirez, Seeger, and Sossa \cite{ramirez-seeger-sossa} prove the following commutation principle.

\begin{theorem}\label{ramirez theorem}
Let $\V$ be a Euclidean Jordan algebra,  $\Omega$ be a spectral set in $\V$, and $F : \V \rightarrow \R$ be 
a spectral function. 
Let  $\Theta:\V\rightarrow R$ be  Fr{\'e}chet differentiable. If $a$ is a local minimizer
 of
        \begin{equation}
                \min_{\Omega}\, \Theta(x) + F(x),
        \end{equation}
        then $a$ and $\Theta^{\prime}(a)$ operator commute in $\V$.
\end{theorem}

A number of important and interesting applications are mentioned in \cite{ramirez-seeger-sossa}. 
The proof of the above result in \cite{ramirez-seeger-sossa} is somewhat intricate, deep, and long. In our paper we 
extend the above result by assuming only the automorphism invariance of $\Omega$ and $F$, and at the same time provide    (perhaps) a simpler and shorter 
proof. To elaborate, recall that an (algebra) automorphism on $\V$ is an invertible linear transformation 
on $\V$ that preserves the Jordan product. It is known (see \cite{jeong-gowda}, Theorem 2) that spectral sets and functions are invariant under automorphisms, but the converse may fail unless the algebra is either $\Rn$ or simple. 
By defining {\it weakly spectral sets/functions} as those having this automorphism invariance property, we extend the above result of Ramirez, Seeger, and Sossa  as follows. 

\begin{theorem}\label{gowda-jeong theorem}
Let $\V$ be a Euclidean Jordan algebra  and suppose that $\Omega$ in $\V$ and $F : \V \rightarrow \R$ are weakly spectral. 
Let  $\Theta:\V\rightarrow R$ be  Fr{\'e}chet differentiable. If $a$ is a local minimizer
 of
        \begin{equation} \label{eq: Minimization in EJA}
                \min_{\Omega}\, \Theta(x) + F(x).
        \end{equation}
        then $a$ and $\Theta^{\prime}(a)$ operator commute in $\V$.
\end{theorem}

By noting that an Euclidean Jordan algebra is an inner product space and the corresponding automorphism group is a subgroup of the orthogonal group (at least when the algebra is equipped with the canonical inner product), we state a similar result
in the setting of a normal decomposition system. Such a system was introduced by Lewis \cite{lewis} to unify various results in convex analysis on  matrices. A  normal decomposition system is a triple  $(X,\G,\gamma)$ where $X$ is a real inner product space, $\G$ is a (closed) subgroup of the orthogonal group of $X$, and $\gamma:X\rightarrow X$ is a mapping that has properties similar to those of $\lambda(x)$, see Section 4. Our commutation principle on such a system is as follows. 

\begin{theorem}\label{normal decomposition system theorem}
        Let $(X,\G,\gamma)$ be a normal decomposition system. Let $\Omega$ be a convex $\G$-invariant set in $X$, $F : X \rightarrow \R$ be a convex $\G$-invariant function, and $\Theta : X \rightarrow \R$ be Fr{\'e}chet differentiable. Suppose that
        $a$ is a  local minimizer of
        \begin{equation} \label{eq: Minimization in NDS}
                \min_{\Omega}\, \Theta(x) + F(x).
        \end{equation}
        Then $-\Theta^{\prime}(a)$ and $a$ commute in $(X,\, \G,\, \gamma)$.
\end{theorem}

The organization of our paper is as follows. We cover some preliminary material in Section 2.  In Section 3, we define weakly spectral sets and present a proof of Theorem \ref{gowda-jeong theorem}. In Section 4, we  describe normal decomposition systems and present a proof of Theorem \ref{normal decomposition system theorem}.
In the Appendix, we state a structure theorem for the automorphism group of a Euclidean Jordan algebra and show  that 
weakly spectral sets and spectral sets coincide only in an essentially simple algebra.  
\section{Preliminaries}
\subsection{Euclidean Jordan algebras}
{\it Throughout this paper,  $\V$ denotes  a Euclidean Jordan algebra} \cite{faraut-koranyi}. For $x,y\in \V$, we denote their
 inner product  by $\langle x,y\rangle$ and Jordan product  by
$x \circ y$. We let $e$ denote the unit element in $\V$ and
$\V_+:=\{x\circ x:x\in \V\}$
denote the corresponding {\it symmetric cone}. If $\V_1$ and $\V_2$ are two Euclidean Jordan algebras, then, $\V_1\times \V_2$ becomes a Euclidean Jordan algebra under the Jordan  and inner products, defined, respectively by
$(x_1,x_2)\circ (y_1,y_2)=\Big ( x_1\circ y_1,x_2\circ y_2\Big )\quad \mbox{and}\quad \langle (x_1,x_2), (y_1,y_2)\rangle=\langle x_1,y_1\rangle+\langle x_2,y_2\rangle.$
A similar definition is made for a product of several Euclidean Jordan algebras.
Recall that a  Euclidean Jordan algebra $\V$ is  {\it simple} if it is not a direct sum/product of  nonzero Euclidean Jordan algebras (or equivalently, if it does not contain any non-trivial ideal). It is known, see \cite{faraut-koranyi}, that any nonzero Euclidean Jordan algebra is, in a unique way, a direct sum/product of simple Euclidean Jordan algebras. Moreover, every simple algebra is isomorphic to one of the following five algebras:

\begin{itemize}
        \item[$(i)$] the algebra $\mathcal{S}^n$ of $n \times n$ real symmetric matrices,
        \item[$(ii)$] the algebra $\mathcal{H}^n$ of $n \times n$ complex Hermitian matrices,
        \item[$(iii)$] the algebra $\mathcal{Q}^n$ of $n \times n$ quaternion Hermitian matrices,
        \item[$(iv)$] the algebra $\mathcal{O}^3$ of $3 \times 3$ octonian Hermitian matrices,
        \item[$(v)$] the Jordan spin algebra $\mathcal{L}^n$ for $n \geq 3$.
\end{itemize}

\gap

We say that $\V$ is {\it essentially simple} if it is either $ \mathcal{R}^n$ or simple.

An element $c \in \V$ is an {\it idempotent} if $c^2 = c$; it is a {\it primitive idempotent} if it is nonzero and cannot be written as a sum of two nonzero idempotents. We say a finite set $\{ e_{1},\, e_{2},\, \ldots,\, e_{n} \}$ of primitive idempotents in $\V$ is a {\it Jordan frame} if
$$ e_{i} \circ e_{j} = 0 \ \mathrm{if} \ i \neq j \quad \mathrm{and} \quad \sum_{i=1}^{n} {e_{i}}=e. $$
It  turns out that the number of elements in any Jordan frame is the same; this common number is called the {\it rank} of $\V$.

\gap

\begin{proposition}[{\it Spectral decomposition theorem} \cite{faraut-koranyi}] \label{prop: spectral decomposition}
        Suppose $\V$ is a Euclidean Jordan algebra of rank $n$. Then, for every $x \in \V$, there exist uniquely determined real numbers $\lambda_1(x),\, \ldots,\, \lambda_n(x)$ (called the {\it eigenvalues} of $x$) and a Jordan frame $\{e_1,\, \ldots,\, e_n\}$ such that
$$x = \lambda_1(x)e_1 + \cdots +\lambda_n(x) e_n.$$
Conversely, given any Jordan frame $\{e_1,\, \ldots,\, e_n\}$ and real numbers $\lambda_1,\lambda_2,\ldots, \lambda_n$, the sum $\lambda_1e_1+\lambda_2e_2+\cdots+\lambda_ne_n$ defines an element of $\V$ whose eigenvalues are $\lambda_1,\lambda_2,\ldots, \lambda_n$.
\end{proposition}

We define the {\it eigenvalue map} $\lambda:\V\rightarrow \Rn$ by 
$$\lambda(x)=\Big ( \lambda_1(x),\lambda_2(x),\ldots, \lambda_n(x)\Big ),$$
where $\lambda_1(x)\geq \lambda_2(x)\geq \ldots\geq  \lambda_n(x).$
This is well-defined and continuous \cite{baes}.\\

We define the {\it trace} of an element $x\in \V$ by $\tr(x):= \lambda_1(x) + \cdots +\lambda_n(x).$ Correspondingly, the {\it canonical (or trace) inner product} on $\V$ is defined by 
$$\langle x,y\rangle_{tr}:=\tr(x\circ y).$$
This defines an inner product on $\V$ that is compatible with the given Jordan structure. With respect to this inner product, the norm of any primitive element is one.

Throughout this paper, for a linear transformation $A:\V\rightarrow \V$ and $x\in \V$, we use, depending on the context,
 both the function notation $A(x)$ as well as the operator notation $Ax$.

Given $a\in \V$, we define the corresponding transformation $L_a:\V\rightarrow \V$ by $L_a(x)=a\circ x$. We say that two elements $a$ and $b$ {\it operator commute} in $\V$ if $L_a\,L_b=L_b\,L_a$. 
We remark that $a$ and $b$ operator commute if and only if there exist a Jordan frame $\{e_1,e_2,\ldots, e_n\}$ and real numbers $a_1,a_2,\ldots, a_n$, $b_1,b_2,\ldots, b_n$ such that 
$$a=a_1e_1+a_2e_2+\cdots+a_ne_n\quad\mbox{and}\quad b=b_1e_1+b_2e_2+\cdots+b_ne_n,$$
see \cite{faraut-koranyi}, Lemma X.2.2. (Note that the $a_i$s and $b_i$s need not be in the decreasing order.) In particular, in $\mathcal{S}^n$ or $\mathcal{H}^n$, operator commutativity reduces to the ordinary 
(matrix) commutativity.

A linear transformation between two Euclidean Jordan algebras is a (Jordan algebra) {\it homomorphism} if it preserves
  Jordan products. If it is also one-to-one and onto, then it is an {\it isomorphism}. If the algebras are the same, we call such an isomorphism an {\it automorphism}. Thus,
a linear transformation $A:\V\rightarrow \V$ is an {\it automorphism} of $\V$ 
 if it is invertible and
$$A(x \circ y) = Ax \circ Ay\quad \forall\,\,x,\, y \in \V.$$
The set of all  automorphisms of $\V$ is denoted by $\Aut(\V)$. 
When $\V$ is a product, say, $\V=\V_1\times \V_2$, for 
$\phi_i\in \Aut(\V_i)$, it is easy to see that $\phi$ defined by  
$$\phi(x):=\Big ( \phi_1(x_1),\phi_2(x_2)\Big ),\quad x=(x_1,x_2)\in \V_1\times \V_2$$
belongs to $\Aut(\V)$. Thus,
$$\Aut(\V_1)\times \Aut(\V_2)\subseteq \Aut(\V_1\times \V_2).$$
A similar statement can be made when $\V$ is a product of several factors. \\
When $\V$ carries the canonical inner product, every automorphism is inner product preserving and so, $\Aut(\V)$ is a closed subgroup of the orthogonal group of $\V$. A linear transformation $D:\V\rightarrow \V$ is a {\it derivation} if 
$$D(x\circ y)=D(x)\circ y+x\circ D(y)\quad \forall\,\,x,\, y \in \V.$$

We recall the following result from \cite{jeong-gowda} (essentially, \cite{faraut-koranyi}, Theorem IV.2.5).

\begin{proposition} \label{prop: jordan frame}
        Let $\V$ be essentially simple. If $\{e_1,\, \ldots, e_n\}$ and $\{e'_1,\, \ldots, e'_n\}$ are any
two Jordan frames in $\V$, then there exists $\phi \in \Aut(\V)$ such that $\phi(e_i) = e'_i$ for all $i = 1,\, \ldots,\, n$.
\end{proposition}

\section{Weakly spectral sets and functions}

\begin{definition}
We say that a set $E$ in $\V$ is {\it weakly spectral} if 
$$A(E)\subseteq E\quad\mbox{for all}\,\,A\in \Aut(\V).$$
A function $F:\V\rightarrow \R$ is said to be {\it weakly spectral} if 
$$F(Ax)=F(x)\,\,\mbox{for all}\,\,x\in \V,\,A\in \Aut(\V).$$
\end{definition}

\begin{remarks} 
${\bf (1)}$ Suppose $E$ is a spectral set, that is, 
$E=\lambda^{-1}(Q)$ for some permutation invariant set $Q$ in $\Rn$. Then,
\begin{equation}\label{spectral property}
x\in E,\,y\in \V,\,\lambda(y)=\lambda(x)\Rightarrow y\in E.
\end{equation}
Now, let $x\in E$ and $A\in \Aut(\V)$. As $A$ maps Jordan frames to Jordan frames, $\lambda(Ax)=\lambda(x)$. From (\ref{spectral property}), $Ax\in E$. This proves that $E$ is weakly spectral. Hence, {\it every spectral set is weakly spectral.}
Now suppose $F:\V\rightarrow \R$ is a spectral function so that for some permutation invariant function $f:\Rn\rightarrow \R$, $F=f\circ \lambda$. It follows that $F(Ax)=f(\lambda(Ax))=f(\lambda(x))=F(x)$ for any $A\in \Aut(\V)$. Thus, $F$ is weakly spectral.
This proves that {\it every spectral function is weakly spectral.}\\
${\bf (2)}$ It has been observed in \cite{jeong-gowda}, Theorem 2, that {\it in any essentially simple algebra, every weakly spectral set is spectral.} The 
 following example shows that weakly spectral sets/functions can be different from spectral sets/functions in general algebras. 
 \\In the
product algebra $\V = \mathcal{R} \times \mathcal{S}^2$, let $\Omega = \mathcal{R}_+ \times \mathcal{S}^2$, and
        $$ x = \Big( 1,\, \big[ \begin{smallmatrix} -1 & 0 \\ 0 & 2 \end{smallmatrix} \big] \Big),\, y = \Big( -1,\, \big[ \begin{smallmatrix} 1 & 0 \\ 0 & 2 \end{smallmatrix} \big] \Big). $$
        Since $x\in \Omega$, $y\not\in \Omega$, and $\lambda(x)=(2,1,-1)^\T=\lambda(y)$,  we see that $\Omega$ cannot be of the form $\lambda^{-1}(Q)$ for any (permutation invariant) set $Q$ in $\R^3$. Thus, $\Omega$ is not a spectral set in $\V$. Now, identity transformation is the only automorphism of $\R$ and any automorphism of $\mathcal{S}^2$ is of the form 
$X\mapsto UXU^\T$ for some orthogonal matrix $U$.  
As $\R$ and $\mathcal{S}^2$ are non-isomorphic Euclidean Jordan algebras, we see (from Proposition 1 in \cite{gowda} or 
Corollary \ref{corollary}  in the Appendix) that 
 automorphisms of $\V$ are of the form
$(t,X)\mapsto (t,UXU^\T)$, for some orthogonal matrix $U$. 
It follows that $\Omega$ is weakly spectral. The characteristic function of $\Omega$ is an example of a weakly spectral function that is not spectral.
\\
${\bf (3)}$ As a consequence of Corollary \ref{corollary} in the Appendix, one can show the following:
Suppose $\V=\V_1\times \V_2\times\cdots\times\V_m$ where $\V_1,\V_2,\ldots, \V_m$ are non-isomorphic simple algebras.
Let $E_i$ be a spectral set in $\V_i$, $i=1,2,\ldots, m$. Then, $E_1\times E_2\times \cdots\times E_m$ is weakly spectral in $\V$. Not every weakly spectral set in $\V$ arises this way: Referring to example given in Remark 2,
$$\{(t,X)\in \Omega: t+\tr(X)=0\}$$
is weakly spectral but not a product of two spectral sets. \\
We also note, as a consequence of Theorem  \ref{theorem on automorphisms}  that a product of (weakly) spectral sets need not be (weakly) spectral. The set $\{0\}\times {\cal S}^2$ in ${\cal S}^2\times {\cal S}^2$ is one such example.
\\${\bf (4)}$ It will be shown  in  Theorem \ref{theorem on weakly spectral sets}  that in any algebra that is not essentially simple, the class of 
weakly spectral sets is strictly larger than the class of spectral sets. This shows that Theorem \ref{gowda-jeong theorem} is applicable to a  wider class of sets/functions than Theorem \ref{ramirez theorem}. 
\end{remarks}

\gap

\noindent {\bf Proof of Theorem \ref{gowda-jeong theorem}.}
        As $a$ is a local minimizer of the problem (\ref{eq: Minimization in EJA}), we have
        $$ \Theta(a) + F(a) \leq \Theta(x) + F(x) \qquad \text{for all} \quad x \in N_a \cap \Omega, $$
        where $N_a$ denotes some open ball around  $a$. Let $u$ and $v$ be arbitrary (but fixed) elements of $\V$. Let $D:=L_u L_v - L_v L_u$, where $L_u(x) := u \circ x$, etc.
       Then, Proposition II.4.1. in \cite{faraut-koranyi} shows that $D$ is a derivation on $\V$; 
hence, as observed in \cite{faraut-koranyi}, p. 36,  
$e^{tD}$ is an automorphism of $\V$ for all $t\in R$. 
Therefore, by the continuity of $t\mapsto e^{tD}a$ and the automorphism invariance of $\Omega$, 
$x:=e^{tD}a\in N_a\cap \Omega$
        for all $t$ close to zero in $\R$.
        Then,
        $$ \Theta(a) + F(a) \leq \Theta(e^{tD}a) + F(e^{tD}a) \,\, \text{for all} \,\, t\, \text{close to } 0. $$
        As $F(e^{tD}a)=F(a)$ by the automorphism invariance of $F$, we see that
        $$ \Theta(a) \leq \Theta(e^{tD}a) \qquad \text{for all} \quad t \text{ near } 0. $$
        This implies that the derivative of $ \Theta(e^{tD}a)$ at $t=0$ is zero. Thus, we  have $ \ip{\Theta^{\prime}(a)}{Da} = 0$. Putting $b:=\Theta^{\prime}(a)$ and recalling $D=L_uL_v-L_vL_u$, we get
        $$ \ip{b}{(L_u L_v - L_v L_u)(a)} = 0. $$
        Since $L_u$ and $L_v$ are self-adjoint, the above expression can be rewritten as
        $$ \ip{v \circ a}{u \circ b} - \ip{u \circ a}{v \circ b} = 0. $$
        This, upon  rearrangement, leads to $ \ip{(L_b L_a - L_a L_b)u}{v} = 0$.
        As this equation holds for all $u$ and $v$, we see that $L_b L_a = L_a L_b$, proving the operator commutativity of $a$ and $b\,(= \Theta'(a))$ in $\V$.
$\hfill$ $\qed$

\gap

An immediate special case of Theorem \ref{gowda-jeong theorem} is obtained by taking $F=0$: If $\Omega$ is weakly spectral and $\Theta$ is Fr\'{e}chet differentiable, then any local minimizer $a$ of $\min_{\Omega}\,\Theta(x)$ operator commutes with $\Theta^{\prime}(a)$. This can further be specialized by assuming that $\Theta$ is linear, that is, of the form $\Theta(x)=\langle b,x\rangle.$
\\
A number of applications mentioned in \cite{ramirez-seeger-sossa} have analogs for weakly spectral sets and functions. We mention one application that is especially important. 

\begin{theorem}\label{vi}
Suppose $\Omega\subseteq \V$ and $F:\V\rightarrow \R$ are weakly spectral. Let $G:\V\rightarrow \V$ be arbitrary. Consider the {\it variational inequality problem} VI$(G,\Omega,F)$: Find $x^*\in \Omega$ such that
$$ \langle G(x^*),x-x^*\rangle +F(x)-F(x^*)\geq 0\quad\mbox{for all}\,\,x\in \Omega.$$
If $a$ solves VI$(G,\Omega,F)$, then $a$ operator commutes with $G(a)$.
\end{theorem}

\begin{proof} The proof is similar to the one given in \cite{ramirez-seeger-sossa}, Proposition 1.9. If $a$ solves VI$(G,\Omega,F)$, then
$$\langle G(a),x-a\rangle+F(x)-F(a)\geq 0\,\,\mbox{for all}\,\,x\in \Omega.$$
This implies
$$\langle G(a),x\rangle+F(x)\geq \langle G(a),a\rangle+F(a)\,\,\mbox{for all}\,\,x\in \Omega.$$ 
So, $a$ minimizes $\langle G(a),x\rangle+F(x)$ over $\Omega$.
By Theorem \ref{gowda-jeong theorem} applied to $\Theta(x):=\langle G(a),x\rangle$, we see that $a$ commutes with $G(a)$.
\end{proof}

As an illustration of the above result, let $K$ be a closed convex cone in $\V$ 
and $G:\V\rightarrow \V$ be arbitrary.
Consider the {\it cone complementarity problem} CP$(G,K)$ of finding an $x^*\in K$ 
such that 
$$x^*\in K,\, G(x^*)\in K^*,\,\,\mbox{and}\,\,\langle x^*,G(x^*)\rangle=0$$
where $K^*$ is the dual of $K$ defined by $K^*=\{y\in \V:\langle y,x\rangle\geq 0\,\,\mbox{for all}\,\,x\in K\}.$

\begin{corollary}\label{cp}
Suppose  $K$ is weakly spectral. If $a$ solves the cone complementarity problem CP$(G,K)$, then $a$ commutes with $G(a)$.
\end{corollary}

\begin{remark} {\bf (5)} The above corollary yields the  following:
{\it Suppose $K$ is a closed convex cone in $\V$ that is
weakly spectral. If $a\in K$ and $b\in K^*$ satisfy $\langle a,b\rangle=0$, then $a$ and $b$ operator commute. }
Such a  result for $K=\V_+$ (the symmetric cone of $\V$) is well-known, see Proposition 6 in \cite{gowda-sznajder-tao}.
\end{remark}
\section{Normal decomposition systems}

Before giving a proof of Theorem \ref{normal decomposition system theorem}, we briefly recall the definition of a normal decomposition system and mention relevant properties.
\\

\begin{definition}
Let $X$ be a real  inner product space, $\G$ be a closed subgroup of the orthogonal group of $X$, and $\gamma : X \rightarrow X$ be a mapping satisfying the following properties:
\begin{itemize}
        \item [$(a)$] $\gamma$ is $\G$-invariant, that is, $\gamma(Ax) = \gamma(x)$ for all $x \in X$, $A \in \G$.
        \item [$(b)$] For each $x\in X$, there exists $A\in \G$ such that $x=A\gamma(x)$, and
        \item [$(c)$] For all $x,w\in X$, we have $\ip{x}{w} \leq \ip{\gamma(x)}{\gamma(w)}$.
\end{itemize}
Then, $(X,\, \G,\, \gamma)$ is called a {\it normal decomposition system} \cite{lewis}. In such a system, a set $\Omega \subset X$ is said to be {\it $\G$-invariant} if $A(\Omega) \subseteq \Omega$ for all $A \in \G$; a function $F : X \rightarrow \R$ is said to be {\it $\G$-invariant} if $F(Ax) = F(x)$ for all $x \in X$ and $A \in \G$.
\end{definition}

In \cite{lewis}, various results on normal decomposition systems are given. In particular, the following is proved:

\begin{proposition} [\cite{lewis}, Proposition 2.3] \label{prop: lewis}
        In a normal decomposition system, for any two elements $x$ and $w$, we have
        $$ \max_{A \in \G}\, \ip{Ax}{w} = \ip{\gamma(x)}{\gamma(w)}. $$
        Also, $\ip{x}{w} = \ip{\gamma(x)}{\gamma(w)}$ holds for two elements $x$ and $w$ if and only if there exists an $A \in \G$ such that $x = A\gamma(x)$ and $w = A\gamma(w)$.
\end{proposition}

Motivated by the above proposition, we say that $x$ and $w$ {\it commute in $(X,\G, \gamma)$} if there exists an $A \in \G$ such that $x = A\gamma(x)$ and $w = A\gamma(w)$.\\

Now consider an essentially simple Euclidean Jordan algebra $\V$. We assume that $\V$ carries the 
canonical inner product and let $\G = \Aut(\V)$.
Let $\{\mybar{e}_1,\, \ldots,\, \mybar{e}_n\}$ be a fixed Jordan frame in $\V$
(with specified order). Define for any $x\in \V$,
\begin{equation} \label{eq: gamma}
        \gamma(x) := \sum_{i=1}^{n} \lambda_i(x) \mybar{e}_i,
\end{equation}
where $\lambda_i(x)$ are components of $\lambda(x)$. As eigenvalues are preserved under automorphisms, we see that 
$\gamma$ satisfies condition $(a)$ in the definition of  normal decomposition system. 
Since $\V$ is essentially simple, any Jordan frame can be mapped onto any other by an element of $\G$ (by Proposition \ref{prop: jordan frame}). Thus, given any $x\in \V$ with its spectral decomposition  $x = \sum_{1}^{n} \lambda_{i}(x) f_i$, we can find  $A\in \G$ such that $A(\mybar{e_i})=f_i$ for all $i$. Then,
$$ x = A\left( \sum_{i=1}^{n} \lambda_i(x) \mybar{e}_i \right) = A\gamma(x). $$
This verifies condition $(b)$ in the definition of  normal decomposition system.
Finally, for all $x,\, w \in X$, we have the so-called Theobald- von Neumann inequality 
$\ip{x}{w} \leq \ip{\gamma(x)}{\gamma(w)}$, see for example, \cite{lim-kim-faybusovich}, \cite{baes}, or \cite{gowda-tao-cauchy}. Putting all these together,  we have
the following result:

\begin{proposition} \label{lim et al}
        Every essentially simple Euclidean Jordan algebra $\V$ is a normal decomposition system with $X = \V$, $\G = \Aut(\V)$, and $\gamma : \V \rightarrow \V$ defined as in (\ref{eq: gamma}).
\end{proposition}

In this setting, two elements $x,y\in X$ commute if and only if there exists a Jordan frame $\{f_1,f_2,\ldots, f_n\}$ such that
\begin{equation} \label{simultaneous diagonal representation}
x=\sum_{1}^{n} \lambda_{i}(x) f_i\quad\mbox{and}\quad y = \sum_{1}^{n} \lambda_{i}(y) f_i.
\end{equation}
We note that this is stronger than the operator commutativity of $x$ and $y$. For example, in $\V=\R^2$, $x=(1,0)^\T$ and $y=(0,1)^\T$ operator commute but does not commute in the above sense.

\begin{remark}{\bf (6)}
Lim, Kim, and Faybusovich  (\cite{lim-kim-faybusovich}, Corollary 4), show that when $\V$ is
 a simple Euclidean Jordan algebra, $(\V, {\cal K}, \gamma)$ is a normal decomposition system, where ${\cal K}$ is the connected component of identity in $\Aut(\V)$ and $\gamma$ is defined as in (\ref{eq: gamma}).
\end{remark}

       In  \cite{lewis}, Lewis  provides numerous examples of normal decomposition systems. In particular,
the algebras $\Sn$ and $\Hn$ (see Section 2) are normal decomposition systems where $\G$ is the corresponding
automorphism group, $\gamma(X)$ is the diagonal matrix with $\lambda(X)$ as the diagonal. Another example is the space $M_{m,n}$ of all real $m\times n$ matrices with inner product $\langle X,Y\rangle:=\tr(X^{T}Y)$, with $\G$ consisting
 of transformations of the form $X\mapsto UXV^T$, where $U$ and $V$ are orthogonal matrices, and $\gamma(X)$ is an $m\times n$ matrix with diagonal consisting of singular values of $X$ written in the decreasing order and zeros elsewhere.

\gap

\noindent{\bf Proof of Theorem \ref{normal decomposition system theorem}.} Since $a$ is a local minimizer of the problem (\ref{eq: Minimization in NDS}), we have
        $$ \Theta(a) + F(a) \leq \Theta(x) + F(x) \qquad \text{for all} \quad x \in N_a \cap \Omega, $$
        where $N_a$ denotes (some) neighborhood of $a$. Let $A$ be an arbitrary element of $\G$. As $\Omega$ is convex and $\G$-invariant, we have, for all positive $t$ near zero, $(1-t)a + tAa \in N_a \cap \Omega$. Thus,
        $$ \Theta(a) + F(a) \leq \Theta\Big ((1-t)a + tAa\Big ) + F\Big ((1-t)a + tAa\Big ), $$
        for all positive $t$ near $0$. As $F$ is convex and $\G$-invariant,
        $$ F\Big ((1-t)a + tAa\Big ) \leq (1-t)F(a) + tF(Aa) = (1-t)F(a) + tF(a)=F(a); $$
        hence, $$\Theta(a) \leq \Theta\Big ((1-t)a + tAa\Big )$$ for all positive $t$ near $0$. 
This implies that $\ip{\Theta^{\prime}(a)}{Aa-a} \geq 0$, that is, $\ip{\Theta^{\prime}(a)}{Aa} \geq \ip{\Theta^{\prime}(a)}{a}$. Now let $b := - \Theta^{\prime}(a)$ so that $\ip{b}{Aa} \leq \ip{b}{a}$.
        Then, as $A\in \G$ is arbitrary, we have
        $$ \max_{A\in \G} \, \ip{b}{Aa} \leq \ip{b}{a}. $$
        Using Proposition \ref{prop: lewis}, we see that $\ip{\gamma(b)}{\gamma(a)} \leq \ip{b}{a}$. Since the reverse inequality always holds in a normal decomposition system, the above inequality turns into an equality. By Proposition \ref{prop: lewis}, $a$ and $b$ commute in $(X,\, \G,\, \gamma)$.
$\hfill$ $\qed$

\gap

\begin{remark}{\bf (7)} One might ask if the commutativity of $a$ and $-\Theta^{\prime}(a)$ in the above theorem could be replaced by that of $a$ and $\Theta^{\prime}(a)$. To answer this, we consider $X=\R^2$ with the usual inner product, let $\G$ be the set of all $2\times 2$ signed permutation matrices (having exactly one nonzero entry, either $1$ or $-1$, in each row/column), and $\gamma(x)=|x|^\downarrow$ (which is the vector of absolute values of entries of $x$ written in the decreasing order). Let $a=(-1,1)^\T$ and $b=(3,-2)^\T$, $\Omega:=\mbox{convex-hull}\{Aa:A\in \G\}$, $\Theta(x)=\langle b,x\rangle$, and $F=0$. Then, it is easy to see that $a$ minimizes $\Theta$ over $\Omega$ and commutes with $-b$ (which is $-\Theta^{\prime}(a)$), but does not commute with $b$.
\end{remark}

\gap

We now state analogs of Theorem \ref{vi} and Corollary \ref{cp} in normal decomposition systems.

\begin{theorem}\label{vi for nds}
Suppose $\Omega\subseteq X$ and $F:X\rightarrow \R$ are convex and $\G$-invariant. Let $G:X\rightarrow X$ be arbitrary. Consider the variational inequality problem VI$(G,\Omega,F)$ on $X$: Find $x^*\in \Omega$ such that
$$ \langle G(x^*),x-x^*\rangle +F(x)-F(x^*)\geq 0\quad\mbox{for all}\,\,x\in \Omega.$$
If $a$ solves VI$(G,\Omega,F)$, then $a$ operator commutes with $-G(a)$.
\end{theorem}

When $\Omega=K$ is a closed convex cone and $F=0$, we write CP$(G,K)$ for VI$(G,\Omega,F)$.
\begin{corollary}\label{cp for nds}
Suppose  $K$ is closed convex cone in $X$ that is $\G$-invariant and $G:\V\rightarrow \V$ be arbitrary. 
If $a$ solves the cone complementarity problem CP$(G,K)$, then $a$ commutes with $-G(a)$.
\end{corollary}

\begin{remarks}{\bf (8)} The above corollary yields the  following:
{\it Suppose $K$ is a closed convex cone in $X$ that is
$\G$-invariant. If $a\in K$ and $b\in K^*$ satisfy $\langle a,b\rangle=0$, then $a$ and $-b$  commute. }
\\
{\bf (9)} We specialize the above remark to essentially simple algebras. Let $\V$ be such an algebra and let 
$K$ be a spectral cone (which is a closed convex cone that is spectral) in $\V$. If $a\in K$ and $b\in K^*$ satisfy $\langle a,b\rangle=0$, then there exists a Jordan frame 
$\{f_1,f_2,\ldots, f_n\}$ such that
$$a=\sum_{1}^{n} \lambda_{i}(a) f_i,\quad b = \sum_{1}^{n} \lambda_{n+1-i}(b) f_i,\quad\mbox{and}\quad \sum_{1}^{n} \lambda_{i}(a)\, \lambda_{n+1-i}(b)=0.$$
This comes from (\ref{simultaneous diagonal representation}) by noting $-\lambda_i(-b)=\lambda_{n+1-i}(b)$ and $\langle a,b\rangle=0.$\\
{\bf (10)} Another consequence of Remark {\bf (8)} is the following: Suppose
 $(X,\G,\gamma)$ is a normal decomposition system where 
$$\langle \gamma(x),\gamma(y)\rangle=0\Rightarrow x=0\,\,\mbox{or}\,\,y=0.$$
(We note that $M_{m,n}$ and the system considered in Remark {\bf (7)} are such systems.)  {\it If $K$ is a closed convex cone in $X$  
that is $\G$-invariant, then $K=\{0\}$ or $X$.} This can be seen as follows. Suppose $K$ is different from $\{0\}$ and $X$. Let $a$ be a nonzero element in the boundary of $K$. By an application of the supporting hyperplane theorem, we can find a nonzero $b\in K^*$ such that
$\langle a,b\rangle=0$. By Remark {\bf (8)}, $a$ and $-b$ commute, hence, $a=A\gamma(a)$, $-b=A\gamma(-b)$ for some
$A\in \G$. Then, $\langle \gamma(a),\gamma(-b)\rangle=\langle a,-b\rangle=0.$ It follows that $a=0$ or $b=0$ leading to a contradiction.  
\end{remarks}

\section {Appendix}
Here, we describe a result on the 
 automorphism group of a  Euclidean Jordan algebra which is written as a product of simple algebras.
While this result can be deduced from Theorem VI.18 in \cite{koecher}, for completeness, we present a direct (perhaps, elementary) proof. Using this result, we show that a Euclidean Jordan algebra $\V$ is essentially simple if and only if every weakly spectral set in $\V$ is spectral.\\
 
Consider a (general) Euclidean Jordan algebra  $\V$.  We assume that $\V$ is product of 
distinct non-isomorphic simple algebras $\V_1,\V_2,\ldots, \V_m$ (with possible repetitions). Regrouping the factors, we assume that
\begin{equation}\label{V as a product of Vis}
\V=\Big (\V_1\times \V_1\times \cdots\times \V_1\Big )\times \Big ( \V_2\times \V_2\times \cdots\times \V_2\Big )\times \cdots \times \Big (\V_m\times \V_m\times \cdots\times \V_m\Big ).
\end{equation}
By letting $\W_i:=\V_i\times \V_i\times \cdots\times \V_i$, we  write 
\begin{equation}\label{V in terms of Ws}
\V=\W_1\times \W_2\times \cdots\times \W_m.
\end{equation}

\begin{theorem}\label{theorem on automorphisms}
$$\Aut(\V)=\Aut(\W_1)\times \Aut(\W_2)\times \cdots\times \Aut(\W_m).$$ 
Moreover, any automorphism $\phi$ of $\Aut(\W_i)$ has the following form:
$$\phi=\Big (\phi_1,\phi_2,\ldots,\phi_k\Big )\circ \sigma,$$
where $k$ is the number of factors in $\W_i$, $\phi_j\in \Aut(\V_i)$, $j=1,2,\ldots,k$ and $\sigma$ is a $k\times k$ 
permutation matrix. 
\end{theorem}

Note: The explicit form of the automorphism $\phi$ written with a permutation $\sigma$ is:
$$\phi(x)=\Big ( \phi_1(x_{\sigma(1)}),\phi_2(x_{\sigma(2)}),\ldots,\phi_k(x_{\sigma(k)})\Big )\,\,\mbox{for all}\,\,x=(x_1,x_2,\ldots,x_k)\in \V_i\times \V_i\times \cdots\times \V_i.$$

Before giving a proof, we present two lemmas. In what follows, we write $\dim(X)$ for the dimension of a space $X$.

\begin{lemma}\label{homo lemma}
Suppose that $\V$ and $\W$ are simple Euclidean Jordan algebras and  $A:\V\rightarrow \W$ is a non-zero Jordan homomorphism. Then:
\begin{itemize}
\item [(i)] $\dim(\V)\leq \dim(\W)$. 
\item [(ii)]  If $\dim(\V)= \dim(\W)$, $A$ is an isomorphism.
\item [(iii)] If $\dim(\V)< \dim(\W)$, then zero is the only homomorphism from $\W$ to $\V$.
\end{itemize}
\end{lemma}

\begin{proof} $(i)$ Since the kernel of a  homomorphism is an ideal of $\V$ and $\V$  is simple, we see that either $A$ is zero or one-to-one. Since our $A$ is nonzero, its kernel is $\{0\}$; hence it is one-to-one and so
 $\dim(\V)\leq \dim(\W)$.\\ 
$(ii)$ When $\dim(\V)= \dim(\W)$, this $A$ is also onto; hence it is an isomorphism. \\ 
$(iii)$ Assume $\dim(\V)< \dim(\W)$. If there is a nonzero Jordan homomorphism from $\W$ to $\V$, by $(i)$, $\dim(\W)\leq \dim(\V)$.
This is a contradiction. 
\end{proof}

\begin{lemma}\label{individual homos}
Consider a product Euclidean Jordan algebra $\Z=\Z_1\times \Z_2\cdots\times \Z_m$. Let $A:\Z\rightarrow \Z$ be a  linear transformation
 written in the matrix form $A=[\,A_{ij}]$, where $A_{ij}:\Z_j\rightarrow \Z_i$ is linear. If $A$ is a Jordan homomorphism, then 
so is  $A_{ij}$ for any $i,j$. Furthermore,   $A_{ik}^{\T}\,A_{il}=0$ for all $i$ and $k\neq l$.
\end{lemma}

\begin{proof} For $x,y\in \Z$, we have $A(x\circ y)=Ax\circ Ay.$ Taking $x=(0,\ldots,0,x_j,0,\ldots,0)^\T$ and 
$y=(0,\ldots,0,y_j,0,\ldots,0)^\T$, we get, for any $i,j$, $A_{ij}(x_j\circ y_j)=A_{ij}x_j\circ A_{ij}y_j$. This proves that $A_{ij}$ is a homomorphism. Now suppose $k\neq l$ and let $x=(0,\ldots,0,x_k,0,\ldots,0)^\T$ and $y=(0,\ldots,0,y_l,0,\ldots,0)^\T$. Then $A(x\circ y)=Ax\circ Ay$ yields, $0=A_{ik}x_k\circ A_{il}y_l$. This leads to $\langle A_{ik}x_k, A_{il}y_l\rangle=0$ and to $\langle x_k, A_{ik}^{\T}\,A_{il}\,y_l\rangle=0$. As $x_k$ and $y_l$ are arbitrary, we get $A_{ik}^{\T}\,A_{il}=0$.
\end{proof}

\noindent{\bf Proof of Theorem \ref{theorem on automorphisms}.}
We may assume without loss of generality that all algebras involved carry canonical inner products. Since $\Aut(\W_1)\times \Aut(\W_2)\times \cdots\times \Aut(\W_m)\subseteq \Aut(\V)$, it is enough to prove 
 the reverse inclusion.  As the result is obvious for $m=1$, we assume that $m\geq 2$. 
Let  $A\in \Aut(\V)$. Since we are given  $\V$ by (\ref{V as a product of Vis}) and (\ref{V in terms of Ws}),
we  think of $A$ as a matrix of linear transformations $A=[ A_{ij} ],$
where each $A_{ij}$ is a linear transformation from some $\V_k$ to $\V_l$. By partitioning this matrix, we  write 
$A=[{\bf B}_{kl}],$
where (each block) ${\bf B}_{kl}:\W_l\rightarrow \W_k$ is a linear transformation.  The main part of our 
proof consists in showing
\begin{equation}\label{main part}
{\bf B}_{1j}=0 \,\,\mbox{for all}\,\, j\geq 2.
\end{equation}
Once we establish this, the same argument can then be used  for $A^\T$ (which is the inverse of $A$ as we are using the canonical inner product). This results in ${\bf B}_{j1}=0$ for all $j\geq 2$. It then follows that  ${\bf B}_{11}\in \Aut(\W_1)$ and $A$ could be viewed as an element of $\Aut(\W_1)\times \Aut(\W_2\times\W_3\times\cdots\times\W_m).$ We  then invoke the induction principle to see that $A\in \Aut(\W_1)\times \Aut(\W_2)\times \cdots\times \Aut(\W_m).$
\\
Now towards proving (\ref{main part}), we make the following  claims:\\
\noindent{\bf Claim 1:} 
\begin{itemize}
\item [(a)] {\it If for some $k\neq l$, (the off-diagonal block) ${\bf B}_{kl}$ has a nonzero entry, then, $\dim(\V_l)<\dim(\V_k)$ and 
${\bf B}_{lk}=0$.}
\item [(b)] {\it If $A_{ij}$ is a nonzero entry in (a diagonal block) 
${\bf B}_{kk}$, then all other entries in the row/column of $A$ containing $A_{ij}$ are zero, that is, $A_{il}=0$ and $A_{li}=0$ for all $l\neq j$.
}
\end{itemize}
To see $(a)$, suppose $A_{ij}$ is a nonzero entry in ${\bf B}_{kl}$. Then,
$A_{ij}$ from $\V_l$ to $\V_k$ is a nonzero homomorphism (by Lemma \ref{individual homos}). As $\V_l$ and  $\V_k$ are simple and non-isomorphic, by Lemma \ref{homo lemma}, $\dim(\V_l)<\dim(\V_k).$  Lemma \ref{homo lemma} also
 shows that there cannot be a nonzero homomorphism from $\V_k$ to $\V_l$. Thus, every entry in ${\bf B}_{lk}$ is zero.\\
To see $(b)$,  suppose that $A_{ij}$ is a nonzero entry in a diagonal block
${\bf B}_{kk}$. Then, by Lemma \ref{individual homos}, $A_{ij}:\V_k\rightarrow \V_k$ is an isomorphism. From the same lemma,
for $l\neq j$, we have $A_{ij}^{\T}\,A_{il}=0$ and so $A_{il}=0$. Thus, in the row containing $A_{ij}$, all other entries are zero.
By working with the transpose of $A$, we see that the column containing $A_{ij}$ is zero except for the $A_{ij}$th entry. This proves the claim.
\\
\noindent{\bf Claim 2:}
{\it Suppose for some $l$ with $1\leq l\leq m-1$, ${\bf B}_{12},{\bf B}_{23},\ldots,{\bf B}_{l\,l+1}$ are nonzero. Then,
$l<m-1$ and there exists $j> l+1$ such that ${\bf B}_{l+1\,j}$ is nonzero.}\\ 
If this were not true, then either $l=m-1$ or $l<m-1$ and $B_{l+1\,j}=0$ for all $j> l+1$. From Claim 1(a),
$\dim(\V_{l+1})<\dim(\V_{l})<\cdots<\dim(\V_{1}).$ From Lemma 1(iii), ${\bf B}_{l+1\,1},{\bf B}_{l+1\,2},\ldots, {\bf B}_{l+1\,l}$ are all zero. This means that in the matrix with entries ${\bf B}_{ij}$, in the $l+1$ row, all entries except possibly ${\bf B}_{l+1\,l+1}$, are zero. As $A$ is invertible, this lone entry ${\bf B}_{l+1\,l+1}$ cannot be zero. In fact, no row in the matrix
 ${\bf B}_{l+1\,l+1}$ can be zero. By Claim 1(b), each row of ${\bf B}_{l+1\,l+1}$ contains exactly one nonzero entry. This implies that in the square matrix ${\bf B}_{l+1\,l+1}$, each column will also contain exactly one nonzero entry. By Claim 1(b), all columns of ${\bf B}_{l\,l+1}$ will  be zero. This  contradicts our assumption that ${\bf B}_{l\,l+1}$ is nonzero. This proves our claim. 
\\
Now suppose, if possible, (\ref{main part}) is false so that 
${\bf B}_{1j}\neq 0$ for some $j\geq 2$. We may assume, by relabeling, that $j=2$, so
${\bf B}_{12}$ is nonzero.
By Claim 2 (with $l=1$), $2<m$ and there exists $j>2$ such that ${\bf B}_{2j}$ is nonzero. Relabeling, we may assume that $j=3$ so that ${\bf B}_{23}$ is nonzero. We can use Claim 2 again, to see that ${\bf B}_{34}$ is nonzero, etc. Claim 2 allows us to repeat this process; however,  as  $m$ is finite, this cannot continue forever. Thus, we reach a contradiction. Hence, (\ref{main part}) holds and as discussed before, leads to the completion of the proof of the first part of the theorem.
\\We now come to the second part of the theorem. For simplicity, we let $i=1$. We need to describe the matrix $A$ which is now ${\bf B}_{11}$. As $A$ is invertible, each row of ${\bf B}_{11}$ is nonzero.  By Claim 1(b), each row of ${\bf B}_{11}$ contains exactly one nonzero entry which, by Lemma 1, is an automorphism of $\V_1$. (This means that  
 each column of ${\bf B}_{11}$ also has the same property.) Thus, ${\bf B}_{11}$ can be regarded as a permutation of a diagonal matrix of transformations where each diagonal entry is an automorphism of $\V_1$. This gives the stated assertion. 
$\hfill$ $\qed$

\gap

The following is immediate.

\begin{corollary} \label{corollary}
Suppose $\V=\V_1\times\V_2\times \cdots\times\V_m$, where  $\V_1,\ldots,\V_m$ are non-isomorphic simple algebras. Then,
$$\Aut(\V)=\Aut(\V_1)\times \Aut(\V_2)\times \cdots\times \Aut(\V_m).$$
\end{corollary}

As an application of the above results, we prove the following.

\begin{theorem}\label{theorem on weakly spectral sets}
$\V$ is essentially simple if and only if every weakly spectral set in $\V$ is spectral.
\end{theorem}

\begin{proof}
The `only if' part has been observed in \cite{jeong-gowda}, Theorem 2. We prove the `if' part. Suppose, if possible, $\V$ is not essentially simple; let $\V$  be given by (\ref{V as a product of Vis}) and (\ref{V in terms of Ws}). We consider two cases:\\
\noindent{\bf Case 1}: $\V=\W_1\times \W_2\times \cdots\times \W_m$, $m\geq 2$. For $i=1,2,\ldots, m$, let 
 rank$(\W_i)=n_i$, ${\cal P}_i$ denote the set of all primitive idempotents in $\W_i$, and $0$ denote the zero element in any $\W_i$. Since automorphisms map primitive idempotents to primitive idempotents,  ${\cal P}_1$  is invariant under automorphisms of $\W_1$, and so, by Theorem \ref{theorem on automorphisms},  
$\Omega:={\cal P}_1\times \{0\}\times \{0\}\ldots\times \{0\}$ is weakly spectral in $\V$.
Let $c_1\in {\cal P}_1$, $c_2\in {\cal P}_2$, 
$$x=\Big (c_1,0,0,\ldots,0\Big )\quad\mbox{and}\quad y= \Big (0,c_2,0,0\ldots,0\Big ).$$
 As both $x$ and $y$ have eigenvalues $1$ (appearing once) and $0$ (appearing $n_1+n_2+\ldots+n_m-1$ times), we see that
 $\lambda(x)=\lambda(y)$. However, $x\in \Omega$ while $y\not\in \Omega$. Thus, $\Omega$ cannot be of the form $\lambda^{-1}(Q)$ for any (permutation invariant) set $Q$.\\
\noindent{\bf Case 2}: $\V=\W_1=\V_1\times \V_1\times \cdots\times \V_1$, where $\V_1$ is a simple algebra of rank at least 2 and the number of factors in this product, say, $m$, is more than one.  Let $n=rank(\V_1)$. In $\R^n$, let $s_i$ denote the coordinate vector containing $1$ in its $i$th slot and zeros elsewhere; let $Q=\{s_1,s_2,\ldots, s_n\}$. As $Q$ is permutation invariant, 
the set ${\cal P}_1:=\lambda^{-1}(Q)$ (which equals the set of all primitive elements in $\V_1$) 
is a spectral set in $\V_1$, where $\lambda:\V_1\rightarrow \Rn$ is the eigenvalue map. As ${\cal P}_1$ is invariant under automorphisms of $\V_1$, an application of  Theorem \ref{theorem on automorphisms} shows that 
the set  $\Omega:={\cal P}_1\times {\cal P}_1\times \cdots\times {\cal P}_1$ is weakly spectral  
in $\V$. We now claim that $\Omega$ is not a spectral set in $\V$. Let $e$ denote the unit vector in $\V_1$ and 
$\{e_1,e_2\ldots, e_n\}$ be a Jordan frame in $\V_1$. Then, $\lambda(e_1)=(1,0,0,\ldots,0)^\T$ and so the vector $x:=(e_1,e_1,\ldots,e_1)$ in $\Omega$ has eigenvalues $1$ (repeated $m$ times) and $0$ (repeated $m(n-1)$ times). When $m\leq n$, let 
$y:=\Big ( e_1+e_2+\cdots+e_m, 0,0,\dots,0\Big )\in \V$. We see that $y\not\in \Omega$ while $\lambda(y)=\lambda(x)$. On the other hand, when $n<m$, we write $m=nk+l$ with $0\leq l<n$ and define $y:=\Big (e,e,\ldots,e,e_1+e_2+\cdots+e_l,0,0\ldots,0\Big )$, where $e$ is repeated $k$ times. We see that $y\not\in \Omega$ while $\lambda(y)=\lambda(x)$. Thus, $\Omega$ is not a spectral set. This completes the proof. 
\end{proof}


\end{document}